\documentclass[12pt, reqno]{amsart}
\usepackage{amsmath, amsthm, amscd, amsfonts, amssymb, graphicx, color}
\usepackage[bookmarksnumbered, colorlinks, plainpages]{hyperref}
\hypersetup{colorlinks=true,linkcolor=red, anchorcolor=green, citecolor=cyan, urlcolor=red, filecolor=magenta, pdftoolbar=true}

\newtheorem{theorem}{Theorem}[section]
\newtheorem{lemma}[theorem]{Lemma}
\newtheorem{question}[theorem]{Question}

\theoremstyle{definition}
\newtheorem{definition}[theorem]{Definition}

\theoremstyle{remark}
\newtheorem{remark}[theorem]{Remark}
\numberwithin{equation}{section}

\begin{document}

\setcounter{page}{1}

\title[On the universal function]{On the universal function for weighted spaces $L^{p}_\mu[0,1],\ p\geq 1$}

\author[M. Grigoryan, T. Grigoryan \MakeLowercase{and} A. Sargsyan]{Martin Grigoryan$^1$, Tigran Grigoryan$^1$ \MakeLowercase{and} Artsrun Sargsyan$^2$$^{*}$}

\address{$^{1}$Department of Physics, Yerevan State University, A. Manoogian 1, 0025 Yerevan, Armenia.}
\email{\textcolor[rgb]{0.00,0.00,0.84}{gmarting@ysu.am;
t.grigoryan@ysu.am}}

\address{$^{2}$CANDLE SRI, Acharyan 31, 0040 Yerevan, Armenia; 
\newline
Russian-Armenian (Slavonic) University, H. Emin 123, 0051 Yerevan, Armenia.}
\email{\textcolor[rgb]{0.00,0.00,0.84}{asargsyan@ysu.am}}





\date{Received: xxxxxx; Revised: yyyyyy; Accepted: zzzzzz.
\newline \indent $^{*}$Corresponding author}

\begin{abstract}
In the paper it is shown that there exist a function $g\in L^1[0,1]$ and a weight function $0<\mu(x)\leq1$, so that $g$ is universal for each classes $L^p_\mu[0,1],\ p\geq 1$ with respect to signs--subseries of its Fourier--Walsh series.
\end{abstract} \maketitle

\section{Introduction and preliminaries}
Let $|E|$ be the Lebesgue measure of a measurable set $E\subseteq [0,1]$, $\chi_E(x)$ -- its characteristic function, $L^{p}(E)\ (p>0)$ -- the class of all those measurable functions on $E$ that satisfy the condition $\int_E|f(x)|^pdx<+\infty$, $L^p_\mu[0,1]$ (weighted space) -- the class of all those measurable functions on $[0,1]$ that satisfy the condition $\int_0^1|f(x)|^p\mu (x)dx<+\infty$, where $0<\mu(x)\leq1$ is a weight function, and $\{\varphi_k\}$ -- a complete orthonormal system in $L^2[0,1]$.

\begin{definition}
Let $0<\mu(x)\leq1,$ be a measurable on $[0,1]$ function. We say that a function $g\in L^1[0, 1]$  is universal for a class $L^p_\mu [0,1]$ with respect to signs--subseries of its Fourier series by the system $\{\varphi_k\}$, if for each function $f\in L^p_\mu [0,1]$ one can choose numbers $\delta_{k}=\pm1, 0$ so that the series 
\begin{equation*}
\sum_{k=0}^{\infty}\delta_{k}c_{k}(g)\varphi_{k}(x),\quad \hbox{with}\quad  c_{k}(g)=\int_{0}^{1}g(x)\varphi_{k}(x)dx,
\end{equation*}
converges to $f$ in $L^p_\mu [0,1]$ metric, i.e.
\begin{equation*}
\lim_{m\to \infty}\int_0^1\left|\sum_{k=0}^m\delta_{k}c_k(g)\varphi_k(x)-f(x)\right|^p\mu (x)dx=0.
\end{equation*}
\end{definition}

Let us recall the definition of the Walsh orthonormal system $\{W_n(x)\}_{n=0}^\infty$. Functions of the Walsh system are defined by means of Rademacher's functions \begin{equation*}
R_n(x)=\hbox{sign}(\sin2^n\pi x ), \quad x\in [0,1], \quad    n=1,2,\dots ,
\end{equation*}
in the following way (see \cite{Gol1987}): $W_0(x)\equiv 1$ and for $n\ge 1$
\begin{equation*}
 W_n(x)=\prod_{i=1}^p R_{k_i+1}(x),
\end{equation*}
where $n=2^{k_1}+2^{k_2}+\dots+2^{k_p}\quad (k_1>k_2>\dots>k_p).$ 

In the present paper the following theorem is proved for the Walsh system:

\begin{theorem}\label{main}
There exist a function $g\in L^1[0,1]$ and a weight function $0<\mu(x)\leq1$, so that $g$ is universal for each class $L^p_\mu[0,1],\ p\geq 1$ with respect to signs--subseries of its Fourier--Walsh series.
\end{theorem}

Moreover, it will be shown that the measure of the set on which $\mu(x)=1$ can be made arbitrarily close to 1, and
the function $g\in L^1[0,1]$ can be choosen to have strictly decreasing Fourier--Walsh coefficients and converging to it by $L^1[0,1]$ norm Fourier--Walsh series.

\begin{remark}
In the proved theorem the weight function $\mu (x)$ cannot be made equal to 1 everywhere in $[0,1]$. Moreover, there does not exist a universal function $g\in L^1[0,1]$ (defined above) for any class $L^p[0,1]$, $p\geq 1$.
\end{remark}

It can be easily shown that the assumption of existence of such universal function simply leads to contradiction. Indeed, if that assumption was true, then for the function $k_0c_{k_0}(g)W_{k_0}(x)$, where $k_0>1$ is any natural number with condition $c_{k_0}(g)\neq 0$, one could find numbers $\delta_{k}=\pm1, 0$ so that

\begin{equation*}
\lim_{m\to \infty}\int_0^1\left|\sum_{k=0}^m\delta_kc_k(g)W_k(x) - k_0c_{k_0}(g)W_{k_0}(x)\right|^pdx=0.
\end{equation*}
Hence, we would simply get a contradiction: $\delta_{k_0}=k_0>1$. 

Existences of functions, which are universal in different senses, were considered by mathematicians since the beginning of the 20-th century. The first type of universal function was considered by G. Birkhoff \cite{Bir1929} in 1929. He proved, that there exists an entire function $g(z)$, which is universal with respect to translations, i.e. for every entire function $f(z)$ and for each number $r>0$ there exists a growing sequence of natural numbers $\{n_k\}_{k=1}^\infty$, so that the sequence $\{g(z+n_k)\}_{k=1}^\infty$ uniformly converges to $f(z)$ on $|z|\leq r$. In 1952 G. MacLane \cite{Mac1952} proved a similar result for another type of universality, namely, there exists an entire function $g(z)$, which is universal with respect to derivatives, i.e. for every entire function $f(z)$ and for each number $r>0$ there exists a growing sequence of natural numbers $\{n_k\}_{k=1}^\infty$, so that the sequence $\{g^{(n_k)}(z)\}_{k=1}^\infty$ uniformly converges to $f(z)$ on $|z|\leq r$. Further, in 1975 S. Voronin \cite{Vor1975} proved the universality theorem for the Riemann zeta function $\zeta(s)$, which states that any nonvanishing analytic function can be approximated uniformly by certain purely imaginary shifts of the zeta function in the critical strip, namely, if $0<r<\frac{1}{4}$ and $g(s)$ is a nonvanishing continuous function on the disk $|s|\leq r$, that is analytic in the interior, then for any $\varepsilon>0$, there exists such a positive real number $\tau$ that
\begin{equation*}
\max_{|s|\leq r}\big|g(s)-\zeta(s+3/4+i\tau)\big|<\varepsilon .
\end{equation*}

In 1987 K. Grosse--Erdman \cite{Gros1987} proved the existence of infinitely differentiable function with universal Taylor expansion, namely, there exists a function $g\in C^{\infty}(\mathcal{R})$ with $g(0)=0$, such that for every function $f\in C(\mathcal{R})$ with $f(0)=0$ and for each number $r>0$ there exists a growing sequence of natural numbers $\{n_k\}_{k=1}^\infty$, so that the sequence
\begin{equation*}
S_{n_k}(g,0)=\sum_{m=1}^{n_k} {g^{(m)}(0)\over m!}x^{m}
\end{equation*}
uniformly converges to $f(x)$ on $|x|\leq r$.

In papers \cite{GrigSar2016} and \cite{GrigSar2016a} authores studied existances of universal functions for classes $L^p[0,1],\ p\in(0,1)$ with respect to signs--subseries of Fourier--Walsh series and signs of Fourier--Walsh coefficients, respectively. In particular, it was shown in \cite{GrigSar2016} that for each number $p\in (0,1)$ one can construct a function from $L^{1}[0,1]$ with convergent in $L^{1}[0,1]$ Fourier--Walsh series having decreasing coefficients, which is universal for the class $L^{p}[0,1]$with respect to signs--subseries of Fourier--Walsh series. 

Note that the definition of function universality which we gave above could be done in therms of Fourier series universality in corresponding sense. The topic of universal series existance (in the common sense, with respect to rearrangements, partial series, signs of coefficients and etc.) in various classical orthogonal systems was also invevestigated intensively. The most general results were obtained by D. Menshov \cite{Men1964}, A. Talalyan \cite{Tal1960}, P. Ulyanov \cite{Ul1972} and their disciples (see \cite{Ol1968}--\cite{Epis2006}). 
    
Regarding to the result of the present paper the following questions arise, the answer to which is unknown yet:

\begin{question}
Is the theorem \ref{main} true for other orthonormal systems (trigonometric system, Franklin system and etc.)?
\end{question}

\begin{question}
Is it possible to acheive universality with respect to signs of Fourier--Walsh coefficients  (i.e. exclude $0$ values from the sequence $\delta_k$) in theorem \ref{main}?
\end{question}

\section{Main lemmas}

Let us start from known properties of the Walsh system, which will be used during the proofs. It is known (see \cite{Gol1987}) that for each natural number $m$

\begin{equation}
\sum_{k=0}^{2^m-1}W_k(x)=
 \begin{cases}
 2^m, &\hbox{when} \quad  x\in[0, 2^{-m}),
 \\
  0, &\hbox{when} \quad  x\in (2^{-m}, 1],
 \end{cases}\label{1.a}
\end{equation}
 and, consequently,

\begin{equation*}
\sum_{k=2^{m}}^{2^{m+1}-1}W_k(x)=
 \begin{cases}
 2^m, &\hbox{when} \quad  x\in [0, 2^{-m-1}),
 \\
 -2^m, &\hbox{when} \quad  x\in(2^{-m-1}, 2^{-m}),
 \\
 0, &\hbox{when} \quad  x\in (2^{-m}, 1],
 \end{cases}
\end{equation*}
thus, for each $p>0$ we have
\begin{equation}
\int_0^1\left|\sum_{k=2^{m}}^{2^{m+1}-1}W_k(x)\right|^pdx=2^{m(p-1)}.\label{1.b}
\end{equation}

Let 
\begin{equation*}
\|\cdot\|_{L^p(E)}=\left(\int_E|\cdot|^pdx\right)^{\frac{1}{p}}\quad \hbox{and} \quad \|\cdot\|_{L^p_\mu [0,1]}=\left(\int_0^1|\cdot|^p\mu (x)dx\right)^{\frac{1}{p}},
\end{equation*}
where $p\geq 1$, $E\subseteq [0,1]$ and $0<\mu(x)\leq 1$, be the norms of spaces $L^p(E)$ and $L^p_\mu[0,1]$, respectively.
Obviously, for any natural number $M\in[2^m,2^{m+1})$ and numbers $\left\{a_k\right\}^{2^{m+1}-1}_{k=2^m}$
\begin{equation}
\left\|\sum_{k=2^{m}}^M a_k W_k\right\|_{L^1[0,1]}\leq\left\| \sum_{k=2^{m}}^{2^{m+1}-1} a_k W_k\right\|_{L^2[0,1]}.\label{2}
\end{equation}

Note also that the basicity of the Walsh system in spaces $L^p[0,1],\ p> 1$ provides the existence of a constant $C_p>0$, so that for each function $f\in L^p[0,1]$ the following inequality holds:
\begin{equation}
\|S_k(f)\|_{L^p[0,1]}\leq C_p\|f\|_{L^p[0,1]},\quad \forall k\in\mathbb{N},\label{3}
\end{equation}
where $\{S_k(f)\}$ are partial sums of its expansion by the Walsh system \cite{Gol1987}.

In the paper we use the following lemma, which was proved in \cite{Nav1994}:

\begin{lemma}\label{lem1}
For each dyadic interval $\Delta=\left[\frac{i}{2^K},\frac{i+1}{2^K}\right]$, $0\leq i<2^K$, and for every natural number $M>K,$ such that $\frac{M-K}{2}$ is a whole number, there exists a polynomial in the Walsh system
\begin{equation*}
H(x)=\sum_{k=2^M}^{2^{M+1}-1}a_kW_k(x),
\end{equation*}
so that

1) $|a_k|=2^{-\frac{M+K}{2}},\quad$ when  $\quad 2^M\leq k<2^{M+1}$,

2) $H(x)=-1,\quad$ if $\quad x\in E_1,\ |E_1|=\frac{1}{2}|\Delta|$,

3) $H(x)=1,\quad$ if $\quad x\in E_2,\ |E_2|=\frac{1}{2}|\Delta|$,

4) $H(x)=0,\quad$ if $\quad x\not\in \Delta$.\newline
where $E_1$ and $E_2$ are finite unions of dyadic intervals.
\end{lemma}

One of the main building blocks in the proof of the theorem \ref{main} is Lemma \ref{lem3} which is proved by the help of Lemma \ref{lem2}.

\begin{lemma}\label{lem2}
Let $p>1$, $n_0$ be some natural number and $\Delta\subset\left[0,1\right]$ be a dyadic interval, then for any numbers $0<\varepsilon<1$, $l\neq 0$ and natural number $q $ there exist a measurable set $E_q\subset\Delta$  with measure $|E_q| =(1-2^{-q})|\Delta|$ and polynomials 
\begin{equation*}
P_q(x)=\sum_{k=2^{n_0}}^{{2^{n_q}}-1}a_kW_k(x)\quad\hbox{and}\quad
H_q(x)=\sum_{k=2^{n_0}}^{{2^{n_q}}-1}\delta_ka_kW_k(x),\quad \delta_{k}=\pm1,0,
\end{equation*}
in the Walsh system, so that $H_q(x) = 0 $ outside $\Delta$,
\begin{equation*}
0<a_{k+1}\leq a_k< \varepsilon \quad \hbox{when} \quad k\in[2^{n_0}, 2^{n_q}-1),\leqno1)
\end{equation*} 
\begin{equation*}
\|l\chi_\Delta-H_q\|_{L^p(E_q)}=0,\leqno2)
\end{equation*}
\begin{equation*}
\max_{2^{n_0}\leq M < 2^{n_q}}\left\|\sum_{k=2^{n_0}}^M \delta_ka_k W_k\right\|_{L^p[0,1]}<2^qC|l||\Delta|^{\frac{1}{p}},\leqno3)
\end{equation*}
where $C$ is a constant defined by the space $L^p[0,1]$, and
\begin{equation*}
\max_{2^{n_0}\leq M < 2^{n_q}}\left\| \sum_{k=2^{n_0}}^Ma_k W_k\right\|_{L^1[0,1]}<\varepsilon.\leqno4)
\end{equation*}
\end{lemma}

\begin{proof}
The proof is performed using mathematical induction with respect to the number $q$. Let $\Delta =\left[\frac{i}{2^K},\frac{i+1}{2^K}\right]\subset\left[0,1\right]$. Choosing a natural number $K_1>K$ such that
\begin{equation}
|l|2^{-\frac{K_1+1}{2}}<\frac{\varepsilon}{2},\label{4}
\end{equation}
we present the interval $\Delta$ in the form of union of disjoint dyadic intervals
\begin{equation*}
\Delta=\bigcup_{i=1}^{N_1}\Delta_i^{(1)}
\end{equation*}
with measure $\big|\Delta_i^{(1)}\big|=2^{-K_1-1},\ i=\overline{1,N_1}.$
Obviously, $N_1=2^{K_1-K+1}$.

By denoting $K_{0}^{(1)}\equiv n_0-1$, for each natural number $i\in[1,N_1]$ we choose a natural number $K_i^{(1)}>K_{i-1}^{(1)}\ \bigl(K_1^{(1)}>K_1\bigr)$ such that the following conditions take place: 

\smallskip
a) $\frac{K_i^{(1)}-K_1-1}{2}$ is a whole number, 

\smallskip
b) $(K_i^{(1)}-K_{i-1}^{(1)})|l|2^{-\frac{K_i^{(1)}+K_1+1}{2}}<\frac{\varepsilon}{4N_1},$

\smallskip
c) $2|l|2^{-\frac{K_i^{(1)}+1}{2}}<\frac{\varepsilon}{2}.$

\smallskip
It immediately follows from \eqref{4} that
\begin{equation}
|l|2^{-\frac{K_1^{(1)}+K_1+1}{2}}<\varepsilon.\label{5}
\end{equation}

By successively applying  lemma \ref{lem1} for each interval $\Delta_i^{(1)}\ (i=\overline{1,N_1})$ and corresponding number $K_i^{(1)}$, we can find polynomials in the Walsh system 
\begin{equation}
\overline H_i^{(1)}(x)=\sum_{k=2^{K_i^{(1)}}}^{2^{K_i^{(1)}+1}-1}\bar a_kW_k(x),\quad i=\overline{1,N_1}\label{6}
\end{equation}
such that
\begin{equation}
|\bar a_k| =|l|2^{-\frac{K_i^{(1)} + K_1+1}{2}},\quad \hbox{when}\quad k\in\bigl[2^{K_i^{(1)}}, 2^{K_i^{(1)}+1}\bigr),\label{7}
\end{equation}
\begin{equation}
\overline H_i^{(1)}(x)=
\begin{cases}
-l, & \hbox{for} \quad x\in \widetilde{E_i}^{(1)}\subset \Delta_i^{(1)},\quad \big|\widetilde{E_i}^{(1)}\big|=\frac{1}{2}\big|\Delta_i^{(1)}\big|,
\\
l, & \hbox{for} \quad x\in \widetilde{\widetilde{E_i}}^{(1)}\subset \Delta_i^{(1)},\quad \big|\widetilde{\widetilde{E_i}}^{(1)}\big|=\frac{1}{2}\big|\Delta_i^{(1)}\big|,
\\
0, & \hbox{for} \quad x\notin \Delta_i^{(1)}.
\end{cases}\label{8}
\end{equation}

Hence, by denoting
\begin{equation}
H_1(x)=\sum_{i=1}^{N_1}\overline H_i^{(1)}(x),\label{9}
\end{equation}
we get
\begin{equation}
H_1(x)=
\begin{cases}
-l, & \hbox{for} \quad x\in \widetilde E_1\subset \Delta,\quad  \big|\widetilde E_1\big|= \frac{|\Delta|}{2},
\\
l, & \hbox{for} \quad x\in \Delta\setminus \widetilde E_1,
\\
0, & \hbox{for} \quad x\notin \Delta.
\end{cases}\label{10}
\end{equation}

As the polynomial $\overline H_i^{(1)}(x)$ is a linear combination of Walsh functions from $K_i^{(1)}$ group, it is clear, that the set $\widetilde E_1$ can be presented as a union of certain $N_2$ number of disjoint dyadic intervals
\begin{equation*}
\widetilde E_1=\bigcup_{i=1}^{N_2}\Delta_i^{(2)}
\end{equation*}
with measure $\big|\Delta_i^{(2)}\big|=2^{-K_{N_1}^{(1)} -1},\ i=\overline{1,N_2}.$

By defining
\begin{equation}
E_1 = \Delta\setminus \widetilde E_1\label{11}
\end{equation}
and 
\begin{equation}
\begin{cases}
\bar a_k = |l|2^{-\frac{K_i^{(1)}+K_1+1}{2}},\quad \hbox{when}\quad k\in \bigl[2^{K_{i-1}^{(1)}+1}, 2^{K_i^{(1)}}\bigl),\quad i\in[1,N_1],
\\
\bar\delta_k =
\begin{cases}
0, & \hbox{when} \quad k\in \bigl[2^{K_{i-1}^{(1)}+1}, 2^{K_i^{(1)}}\bigr)
\\
1, & \hbox{when} \quad k\in\bigl[2^{K_i^{(1)}}, 2^{K_i^{(1)}+1}\bigr)
\end{cases},\quad i\in[1,N_1],
\\
a_k=|\bar a_k|,\quad \delta_k=\bar\delta_k \cdot \frac{\bar a_k}{|\bar a_k|},\quad \hbox{when}\quad k\in\bigl[2^{n_0},2^{K_{N_1}^{(1)}+1}\bigr),
\end{cases}\label{12}
\end{equation}
let us verify that the set $E_1$ and polynomials
\begin{equation*}
P_1(x)=\sum_{k=2^{n_0}}^{{2^{K_{N_1}^{(1)}+1}}-1}a_kW_k(x)\quad\hbox{and}\quad
H_1(x)=\sum_{k=2^{n_0}}^{{2^{K_{N_1}^{(1)}+1}}-1}\delta_ka_kW_k(x),\quad \delta_k=\pm1, 0
\end{equation*}
satisfy all statements of lemma \ref{lem2} for $q=1$. Indeed, by using \eqref{10} and \eqref{11} we obtain $| E_1|=(1-2^{-1})|\Delta|$. The statement 1) follows from \eqref{5}, \eqref{7}, \eqref{12} and from monotonicity of numbers $K^{(1)}_i\ (i=\overline{1,N_1})$. The statement 2) immediately follows from \eqref{10} and \eqref{11}. To prove statements 3) and 4) we present the natural number $M\in \bigl[2^{n_0}, 2^{K_{N_1}^{(1)}+1}\bigr)$ in the form $M = 2^{\bar n} +s,\ s\in[0, 2^{\bar n})$, where $\bar n\in \bigl(K_{m-1}^{(1)}, K_{m}^{(1)}\bigr]$ for some $m\in[1,N_1]$. Since intervals $\Delta_i^{(1)}\ (i=\overline{1,N_1})$ are disjoint, by using \eqref{3}, \eqref{6}, \eqref{8}--\eqref{12} we have

\begin{equation*}
\left\| \sum_{k=2^{n_0}}^M \delta_ka_k W_k\right\|_{L^p[0,1]}\leq 
\left\|\sum_{i=1}^{m-1}\overline H_i^{(1)}\right\|_{L^p[0,1]} + \left\| \sum_{k=2^{\bar n}}^{2^{\bar n}+s}\delta_k a_kW_k\right\|_{L^p[0,1]}\leq
\end{equation*}
\begin{equation*}
\leq\|H_1\|_{L^p[0,1]} 
+C_p\bigl\|\overline H_m^{(1)}\bigr\|_{L^p[0,1]}\leq
\end{equation*}
\begin{equation*}
=\left(|l|^p|E_1|+ |l|^p|\widetilde E_1|\right)^{\frac{1}{p}}+C_p|l|\big| \Delta_m^{(1)}\big|^{\frac{1}{p}}<2C|l||\Delta|^{\frac{1}{p}},
\end{equation*}
where $C=C_p+1$.

Further, for each natural number $n\in \bigl[n_0, K_{N_1}^{(1)}\bigr]$ we denote $b_n=a_k,\ k\in[2^n, 2^{n+1})$ (coefficients $a_k$ of Walsh functions from n--th group are equal in $H_1(x)$). Taking into account \eqref{1.b}, \eqref{2}, \eqref{4}, \eqref{7}, \eqref{12} and b) condition for numbers $K_{i}^{(1)}\ (i=\overline{1,N_1})$, we get

\begin{equation*}
\sum_{n={n_0}}^{K_{N_1}^{(1)}}b_n=\sum_{i=1}^{N_1} \sum_{n={K_{i-1}^{(1)}+1}}^{{K_{i}^{(1)}}} b_n=\sum_{i=1}^{N_1}\bigl(K_{i}^{(1)}-K_{i-1}^{(1)}\bigr)|l|2^{-\frac{K_{i}^{(1)}+K_1+1}{2}}<\frac{\varepsilon}{4},
\end{equation*}
\begin{equation*}
\left\| \sum_{k=2^{n_0}}^Ma_k W_k\right\|_{L^1[0,1]}\leq\sum_{n={n_0}}^{\bar n-1}b_n + \left\|\sum_{k=2^{\bar n}}^{2^{\bar n}+s} a_k W_k\right\|_{L^1[0,1]}\leq
\end{equation*}
\begin{equation*}
\leq\sum_{n={n_0}}^{K_{N_1}^{(1)}}b_n+
\left\|\sum_{k=2^{\bar n}}^{2^{\bar n+1}-1} b_{\bar n} W_k\right\|_{L^2[0,1]}<\frac{\varepsilon}{4}+|l|2^{-\frac{K_{m}^{(1)}+K_1+1}{2}}2^{\frac{\bar n}{2}}<\varepsilon,
\end{equation*}
which proves the statement 4) of lemma \ref{lem2}.

Assume, that for $q>1$ natural numbers 
\begin{equation*}
K^{(1)}_1<\dots<K^{(1)}_{N_{1}}<\dots<K^{(q-1)}_1<\dots<K^{(q-1)}_{N_{q-1}},\end{equation*}
sets 
\begin{equation*}
\widetilde E_{q-1}\subset \Delta\quad \hbox{and}\quad E_{q-1}=\Delta\setminus \widetilde E_{q-1}
\end{equation*}
and polynomials
\begin{equation*}
P_{q-1}(x)=\sum_{k=2^{n_0}}^{{2^{K_{N_{q-1}}^{(q-1)}+1}}-1}a_kW_k(x),
\end{equation*}
\begin{equation*}
H_{q-1}(x)=\sum_{k=2^{n_0}}^{{2^{K_{N_{q-1}}^{(q-1)}+1}}-1}\delta_ka_kW_k(x),\quad \delta_k=\pm1, 0
\end{equation*}
are already chosen to satisfy the conditions

\smallskip
$a^\prime$) $\frac{K_i^{(\nu)}-K_{N_{\nu-1}}^{(\nu-1)}-1}{2}$ is a whole number
$\bigr(K_{N_{0}}^{(0)}\equiv K_1\bigl),$ 

\smallskip
$b^\prime$) $\bigl(K_i^{(\nu)}-K_{i-1}^{(\nu)}\bigr)2^{(\nu-1)}|l|2^{-\frac{K_i^{(\nu)}+K_{N_{\nu-1}}^{(\nu-1)}+1}{2}}<\frac{\varepsilon}{2^{\nu +1}N_{\nu}},$

\smallskip
$c^\prime$) $2^{\nu}|l|2^{-\frac{K_i^{(\nu)}+1}{2}}<\frac{\varepsilon}{2},$\newline

\smallskip
\begin{equation}
a_k = 2^{\nu-1}|l|2^{-\frac{K_i^{(\nu)}+K_{N_{\nu-1}}^{(\nu-1)}+1}{2}}\quad \hbox{for}\quad k\in \bigl[2^{K_{i-1}^{(\nu)}+1}, 2^{K_i^{(\nu)}+1}\bigl),\label{13}
\end{equation}
\begin{equation*}
K_0^{(\nu)}\equiv 
\begin{cases}
K_{N_{\nu-1}}^{(\nu-1)}, & \hbox{if} \quad \nu>1,
\\
n_0-1, & \hbox{if} \quad \nu=1,
\end{cases}
\end{equation*}
for any natural numbers $i\in[1,N_{\nu}]$ and $\nu\in [1, q-1]$. Besides,
\begin{equation}
\sum_{n={n_0}}^{K_{N_{q-1}}^{(q-1)}} b_n<\sum_{k=1}^{q-1}\frac{\varepsilon}{2^{k+1}},\quad\hbox{where}\quad b_n\equiv a_k,\ k\in[2^n, 2^{n+1}),\label{14}
\end{equation}
\begin{equation}
H_{q-1}(x)=
\begin{cases}
-(2^{q-1}-1)l,& \hbox{for} \quad x\in \widetilde E_{q-1},
\\
l, & \hbox{for} \quad x\in E_{q-1},
\\
0,& \hbox{for} \quad x\notin \Delta,
\end{cases}\label{15}
\end{equation}

\begin{equation}
\big|\widetilde E_{q-1}\big|= 2^{-q+1}|\Delta|\quad \hbox{and}\quad \big|E_{q-1}\big| = \bigr(1-2^{-q+1}\bigr)|\Delta|\label{16}
\end{equation}
and the set  $\widetilde E_{q-1}$ can be presented as a union of certain $N_q$ number of disjoint dyadic intervals
\begin{equation}
\widetilde E_{q-1}=\bigcup_{i=1}^{N_{q}}\Delta_i^{(q)}\label{17}
\end{equation}
with measure $\big|\Delta_i^{(q)}\big|=2^{-K^{(q-1)}_{N_{q-1}}-1},\ i=\overline{1,N_q}.$

For each natural number $i\in[1,N_{q}]$ we choose a natural number $K_i^{(q)}>K_{i-1}^{(q)}$ $\bigl(K_{0}^{(q)}\equiv K_{N_{q-1}}^{(q-1)}\bigr)$ such that the following conditions hold: 

\smallskip
$a^{\prime\prime})$ $\frac{K_i^{(q)}-K_{N_{q-1}}^{(q-1)}-1}{2}$ is a whole number, 

\smallskip
$b^{\prime\prime})$ $\bigl(K_i^{(q)}-K_{i-1}^{(q)}\bigr)2^{(q-1)}|l|2^{-\frac{K_i^{(q)}+K_{N_{q-1}}^{(q-1)}+1}{2}}<\frac{\varepsilon}{2^{q+1}N_{q}},$

\smallskip
$c^{\prime\prime})$ $2^{q}|l|2^{-\frac{K_i^{(q)}+1}{2}}<\frac{\varepsilon}{2}.$

\smallskip
By successively applying  lemma \ref{lem1} for each interval $\Delta_i^{(q)}\subset \widetilde E_{q-1}\ (i=\overline{1,N_{q}})$ and corresponding number $K_i^{(q)}$, we can find polynomials in the Walsh system 
\begin{equation}
\overline H_i^{(q)}(x)=\sum_{k=2^{K_i^{(q)}}}^{2^{K_i^{(q)}+1}-1}\bar a_kW_k(x),\quad i=\overline{1,N_{q}},\label{18}
\end{equation}
such that
\begin{equation}
|\bar a_k|= 2^{q-1}|l|2^{-\frac{K_i^{(q)} + K_{N_{q-1}}^{(q-1)}+1}{2}},\quad  \hbox{when} \quad k\in \bigl[2^{K_i^{(q)}}, 2^{K_i^{(q)}+1}\bigr),\label{19}
\end{equation}

\begin{equation}
\overline H_i^{(q)}(x)=
\begin{cases}
-2^{q-1}l, & \hbox{for} \quad x\in \widetilde{E_i}^{(q)}\subset \Delta_i^{(q)},\quad \big|\widetilde{E_i}^{(q)}\big|=\frac{1}{2}\big|\Delta_i^{(q)}\big|,
\\
2^{q-1}l, & \hbox{for} \quad x\in \widetilde{\widetilde{E_i}}^{(q)}\subset \Delta_i^{(q)},\quad \big|\widetilde{\widetilde{E_i}}^{(q)}\big| =\frac{1}{2}\big|\Delta_i^{(q)}\big|,
\\
0, & \hbox{for} \quad x\notin \Delta_i^{(q)}.
\end{cases}\label{20}
\end{equation}

Hence, by denoting
\begin{equation}
H_q(x)=H_{q-1}(x)+\sum_{i=1}^{N_{q}}\overline H_i^{(q)}(x)\label{21}
\end{equation}
and taking into account \eqref{15} and \eqref{17}, we obtain
\begin{equation}
H_q(x)=
\begin{cases}
-(2^q-1)l, & \hbox{for} \quad x\in \widetilde E_q\subset \widetilde E_{q-1},\quad  \big|\widetilde E_q\big|=\frac{|\Delta|}{2^q},
\\
l, & \hbox{for} \quad x\in \Delta\setminus \widetilde E_q,
\\
0, & \hbox{for} \quad x\notin \Delta.
\end{cases}\label{22}
\end{equation}

Now, after defining
\begin{equation}
E_q=\Delta\setminus \widetilde E_q\label{23}
\end{equation}
and 
\begin{equation}
\begin{cases}
\bar a_k = 2^{q-1}|l|2^{-\frac{K_i^{(q)}+K_{N_{q-1}}^{(q-1)}+1}{2}},\quad \hbox{when}\quad  k\in \bigl[2^{K_{i-1}^{(q)}+1}, 2^{K_i^{(q)}}\bigr),
\\
\bar\delta_k =
\begin{cases}
0, & \hbox{when} \quad k\in \bigl[2^{K_{i-1}^{(q)}+1}, 2^{K_i^{(q)}}\bigr),
\\
1, & \hbox{when} \quad k\in\bigl[2^{K_i^{(q)}}, 2^{K_i^{(q)}+1}\bigr),
\end{cases}\quad i\in [1, N_{q}],
\\
a_k=|\bar a_k|,\quad \delta_k=\bar\delta_k \cdot \frac{\bar a_k}{|\bar a_k|},\quad\hbox{when}\quad k\in\bigl[2^{n_0},2^{K_{N_q}^{(q)}+1}\bigr),
\end{cases}\label{24}
\end{equation}
let us verify that the set $E_q$ and polynomials 
\begin{equation*}
P_q(x)=\sum_{k=2^{n_0}}^{{2^{n_q}}-1}a_kW_k(x),
\end{equation*}
\begin{equation*}
H_q(x)=\sum_{k=2^{n_0}}^{{2^{n_q}}-1}\delta_ka_kW_k(x),\quad \delta_k=0,\pm1,
\end{equation*}
where $n_q\equiv K_{N_{q}}^{(q)}+1$, satisfy all statements of lemma \ref{lem2}. Indeed, from \eqref{22} and \eqref{23} it follows that $\big|E_q\big|= (1-2^{-q})|\Delta|$. The statement 1) follows from \eqref{5}, \eqref{13}, \eqref{19}, \eqref{24} and from monotonicity of numbers $K^{(\nu)}_i,\ i\in[1,N_{\nu}],\ \nu\in[1,q]$. The statement 2) immediately follows from \eqref{22} and \eqref{23}. To prove statements 3) and 4) we present the natural number $M\in [2^{n_0}, 2^{n_q})$ in the form $M = 2^{\bar n} + s,\ s\in[0, 2^{\bar n})$. Let us consider only the case when $\bar n\in\bigl(K_{N_{q-1}}^{(q-1)}, K_{N_{q}}^{(q)}\bigr]$, since all other cases are under consideration in previous steps of induction. Let $\bar n\in\bigl(K_{m-1}^{(q)}, K_{m}^{(q)}\bigr]$ for some $m\in[1,N_q]$. From \eqref{3}, \eqref{15}--\eqref{18}, \eqref{20} and \eqref{24} we have

\begin{equation*}
\left\| \sum_{k=2^{n_0}}^M \delta_ka_k W_k\right\|_{L^p[0,1]}\leq 
\left\|H_{q-1}+\sum_{i=1}^{m-1}\overline H_i^{(q)}\right\|_{L^p[0,1]}+
\end{equation*}
\begin{equation*}
+ \left\| \sum_{k=2^{\bar n}}^{2^{\bar n}+s}\delta_k a_kW_k\right\|_{L^p[0,1]}\leq\|H_q\|_{L^p[0,1]} 
+C_p\bigl\|\overline H_m^{(q)}\bigr\|_{L^p[0,1]}<
\end{equation*}
\begin{equation*}
<\left(|l|^p\big|E_q\big| + 2^{pq}|l|^p\big|\widetilde E_q\big|\right)^{\frac{1}{p}}+C_p2^{q-1}|l|\big|\Delta_m^{(q)}\big|^{\frac{1}{p}}<2^qC|l||\Delta|^{\frac{1}{p}}
\end{equation*}
$(C=C_p+1)$, which proves the statement 3).

Further, for each natural number $n\in\bigl[n_0,K_{N_{q}}^{(q)}\bigr]$ we denote
\begin{equation*}
b_n\equiv a_k,\quad\hbox{when}\quad k\in[2^n, 2^{n+1}).
\end{equation*}
Taking into account \eqref{1.b}, \eqref{2}, \eqref{14}, \eqref{19}, \eqref{24}, $c^\prime)$ condition for number $K_{N_{q-1}}^{(q-1)}$ and $b^{\prime\prime})$ condition for numbers $K_{i}^{(q)} (i=\overline{1,N_q})$ we get
\begin{equation*}
\left\| \sum_{k=2^{n_0}}^Ma_k W_k\right\|_{L^1[0,1]}\leq\sum_{n={n_0}}^{\bar n-1}b_n + \left\|\sum_{k=2^{\bar n}}^{2^{\bar n}+s} a_k W_k\right\|_{L^1[0,1]}\leq
\end{equation*}
\begin{equation*}
\leq\sum_{n={n_0}}^{K_{N_{q-1}}^{(q-1)}} b_n +\sum_{i=1}^{N_{q}}\sum_{n={K_{i-1}^{(q)}+1}}^{{K_{i}^{(q)}}} b_n+
\left\|\sum_{k=2^{\bar n}}^{2^{\bar n+1}-1} b_{\bar n} W_k\right\|_{L^2[0,1]}\leq
\end{equation*}
\begin{equation*}
\leq\sum_{k=1}^{q-1}\frac{\varepsilon}{2^{k+1}}+\sum_{i=1}^{N_{q}}\bigl(K_{i}^{(q)}-K_{i-1}^{(q)}\bigr)2^{q-1}|l|2^{-\frac{K_{i}^{(q)}+K_{N_{q-1}}^{(q-1)}+1}{2}}+
\end{equation*}
\begin{equation*}
+2^{q-1}|l|2^{-\frac{K_{m}^{(q)}+K_{N_{q-1}}^{(q-1)}+1}{2}}2^{\frac{\bar n}{2}}<\varepsilon,
\end{equation*}
which proves the statement 4).

Lemma \ref{lem2} is proved.
\end{proof}

\begin{lemma}\label{lem3}
Let numbers $p_0>1,\ n_0\in\mathbb{N},\ 0<\varepsilon<1$ and polynomial $f(x)\not \equiv 0$ in the Walsh system be given. Then one can find a measurable set $E_\varepsilon$ with measure $|E_\varepsilon|> 1-\varepsilon$ and  polynomials 
\begin{equation*}
P(x)=\sum_{k=2^{n_{0}}}^{2^{n}-1}a_{k}W_{k}(x)\quad\hbox{and}\quad
H(x)=\sum_{k=2^{n_{0}}}^{2^{n}-1}\delta_{k}a_{k}W_{k}(x),\quad\delta_{k}=0,\pm1,
\end{equation*}
in the Walsh system, which satisfies the following conditions:
\begin{equation*}
0<a_{k+1}<a_{k}<\varepsilon ,\quad k\in\mathit{[\mathit{2}^{n_{0}},2^{n}-1)},\leqno1)
\end{equation*}
\begin{equation*}
\|f(x)-H(x)\|_{L^{p_0}(E_\varepsilon)}<\varepsilon, \leqno2)
\end{equation*}
\begin{equation*}
\max_{2^{n_{0}}\leq M<2^{n}}\left\|\sum_{k=2^{n_{0}}}^{M}\delta_{k}
a_{k}W_{k}(x)\right\|_{L^p(e)}<\|f(x)|\|_{L^p(e)}+\varepsilon\leqno3)
\end{equation*}
for any measurable set $e\subseteq E_\varepsilon$ and $p\in[1,p_0]$,
\begin{equation*}
\max_{2^{n_{0}}\leq M<2^{n}}\left\|\sum_{k=2^{n_{0}}}^{M}a_{k}
W_{k}(x)\right\|_{L^1[0,1]}<\varepsilon.\leqno4)
\end{equation*}
\end{lemma}

\begin{proof}
We choose a natural number $q$, so that
\begin{equation}
2^{-q}<\varepsilon,\label{25}
\end{equation}
and present the function $f(x)$ in the form
\begin{equation*}
f(x) = \sum_{j=1}^{\nu_0} l_j\chi_{\Delta_j}(x),
\end{equation*}
where $l_j\neq 0,\ j=\overline{1,\nu_0}$, and $\left\{\Delta_j\right\}_{j=1}^{\nu_0}$ are disjoint dyadic subintervals of the section $[0,1]$. Without loss of generality we can assume that all these intervals have the same length and are small enough to provide the condition 
\begin{equation}
\max_{1\leq j\leq \nu_0}\bigl\{2^qC|l_j|\big|\Delta_j\big|^{\frac{1}{p}}\bigr\}<\frac{\varepsilon}{2}.\label{26}
\end{equation}

By successively applying lemma \ref{lem2} for each interval $\Delta_j,\ j=\overline{1,\nu_0}$, and taking into account \eqref{25} and \eqref{26}, we can find sets $E_q^{(j)}\subset \Delta_j$ with measure
\begin{equation}
\big|E_q^{(j)}\big|= (1-2^{-q})\big|\Delta_j\big|>(1-\varepsilon)\big|\Delta_j\big|\label{27}
\end{equation}
and polynomials
\begin{equation*}
\bar P_q^{(j)}(x)=\sum_{k=2^{n_{j-1}}}^{{2^{n_j}}-1}\bar a_k^{(j)}W_k(x),
\end{equation*}
\begin{equation*}
\bar H_q^{(j)}(x)=\sum_{k=2^{n_{j-1}}}^{{2^{n_j}}-1}\delta_k^{(j)}\bar a_k^{(j)}W_k(x),\quad \delta_k^{(j)}=\pm 1, 0
\end{equation*}
in the Walsh system, so that $\bar H_q^{(j)}(x)=0$ outside $\Delta_j$,
\begin{equation}
\begin{cases}
0<\bar a_{k+1}^{(1)}\leq \bar a_k^{(1)}<\frac{\varepsilon}{2}, &\hbox{for all}\quad k\in[2^{n_{0}}, 2^{n_1}-1),
\\
0<\bar a_{k+1}^{(j)}\leq \bar a_k^{(j)}<\bar a_{{2^{n_{j-1}}}-1}^{(j-1)}, &\hbox{for all}\quad k\in[2^{n_{j-1}}, 2^{n_j}-1),\quad  j>1,
\end{cases}\label{28}
\end{equation}

\begin{equation}
\left\|l_j\chi_{\Delta_j}- \bar H_q^{(j)}\right\|_{L^{p_0}\bigl(E_q^{(j)}\bigr)}=0, \label{29}
\end{equation}
\begin{equation}
\max_{2^{n_{j-1}}\leq M < 2^{n_j}}\left\| \sum_{k=2^{n_{j-1}}}^M \delta_k^{(j)}\bar a_k^{(j)} W_k\right\|_{L^{p_0}[0,1]}< 2^qC|l_j|\big|\Delta_j\big|^{\frac{1}{p_0}}<\frac{\varepsilon}{2},\label{30}
\end{equation}
\begin{equation}
\max_{2^{n_{j-1}}\leq M < 2^{n_j}}\left\| \sum_{k=2^{n_{j-1}}}^M\bar a_k^{(j)} W_k\right\|_{L^1[0,1]}<\frac{\varepsilon}{2^{j+1}}.\label{31}
\end{equation}

We define a set 
\begin{equation}
E_\varepsilon = \bigcup_{j=1}^{\nu_0}E_q^{(j)}\bigcup\left([0,1]/\cup_{j=1}^{\nu_0}\Delta_j\right)\label{32}
\end{equation}
and polynomials 
\begin{equation*}
\bar P(x) =\sum_{j=1}^{\nu_0}\bar P_q^{(j)}(x)=\sum_{k=2^{n_0}}^{2^{n_{\nu_0}-1}}\bar a_kW_k(x),
\end{equation*}
\begin{equation*}
\bar H(x) =\sum_{j=1}^{\nu_0}\bar H_q^{(j)}(x)=\sum_{k=2^{n_0}}^{2^{n_{\nu_0}-1}}\delta_k\bar a_kW_k(x),
\end{equation*}
where $\bar a_k=\bar a_k^{(j)}$ and $\delta_k=\delta_k^{(j)}$, when $k\in[2^{n_{j-1}},2^{n_j})$. Note that $\bar H_q^{(j)} = 0$ on the set $[0,1]/\cup_{j=1}^{\nu_0}\Delta_j$ (in case it is not empty) for any $j\in[1,\nu_0]$. 

From \eqref{27}--\eqref{29} and \eqref{32} it follows that
\begin{equation*}
\big|E_\varepsilon\big|>1-\varepsilon,
\end{equation*}
\begin{equation}
0<\bar a_{k+1}\leq \bar a_k<\frac{\varepsilon}{2},\quad\hbox{when}\quad k\in[2^{n_0}, 2^{n_{\nu_0}}-1),\label{33}
\end{equation}
\begin{equation}
\left\|f-\bar H\right\|_{L^{p_0}(E_\varepsilon)}\leq\sum_{j=1}^{\nu_0}
\left\|l_j\chi_{\Delta_j}- \bar H_q^{(j)}\right\|_{L^{p_0}\bigl(E_q^{(j)}\bigr)}=0.\label{34}
\end{equation}

Further, let $M$ be a natural number from $[2^{n_0},2^{n_{\nu_0}})$. Then $M\in [2^{n_{m-1}},2^{n_m})$ for some $m\in[1,\nu_0]$. Taking into account \eqref{29}, \eqref{30} and \eqref{32}, for any measurable set $e\subseteq E_\varepsilon$ and $p\in[1,p_0]$ we have
\begin{equation}
\left\| \sum_{k=2^{n_0}}^M \delta_k\bar a_k W_k\right\|_{L^p(e)}\leq\label{35}
\end{equation}
\begin{equation*}
\leq\left\| \sum_{j=1}^{m-1} \bar H_q^{(j)}\right\|_{L^p(e)}+
\left\| \sum_{k=2^{n_{m-1}}}^M \delta_k^{(m)}a_k^{(m)} W_k\right\|_{L^p(e)}\leq
\end{equation*}
\begin{equation*}
\leq \sum_{j=1}^{m-1}\left\| l_j\chi_{\Delta_j}-\bar H_q^{(j)}\right\|_{L^{p_0}\bigl(E_q^{(j)}\bigr)}+\left\|\sum_{j=1}^{m-1}l_j\chi_{\Delta_j}\right\|_{L^p(e)}+
\end{equation*}
\begin{equation*}
+\left\| \sum_{k=2^{n_{m-1}}}^M \delta_k^{(m)}a_k^{(m)} W_k\right\|_{L^{p_0}[0,1]}<\|f\|_{L^p(e)}+\frac{\varepsilon}{2}
\end{equation*}
and, by using \eqref{31}, we obtain
\begin{equation}
\left\|\sum_{k=2^{n_0}}^M \bar a_k W_k\right\|_{L^1[0,1]}\leq\sum_{j=1}^{\nu_0}\max_{2^{n_{j-1}}\leq N < 2^{n_j}}\left\| \sum_{k=2^{n_{j-1}}}^N\bar a_k^{(j)} W_k\right\|_{L^1[0,1]}<\frac{\varepsilon}{2}.\label{36}
\end{equation}

Hence, polynomials $\bar P(x)$ and $\bar H(x)$ satisfy all statements of lemma 3 except for 1). To have strict inequalities between coefficients we choose such a natural number  $N_0$ that
\begin{equation}
2^{-N_0}<\frac{\varepsilon}{2}\label{37}
\end{equation} 
and define polynomials
\begin{equation*}
P(x)=\sum_{k=2^{n_0}}^{2^{n_{\nu_0}}-1}a_kW_k(x)\quad\hbox{and}\quad H(x)=\sum_{k=2^{n_0}}^{2^{n_{\nu_0}}-1}\delta_ka_kW_k(x),
\end{equation*}
where
\begin{equation}
a_k = \bar a_k + 2^{-(N_0+k)}.\label{38}
\end{equation}

It is not hard to verify  that polynomials $P(x)$ and $H(x)$ satisfy all statements of lemma \ref{lem3}. Indeed, the statement 1) immediately follows from \eqref{33}, \eqref{37} and \eqref{38}.  Further, considering \eqref{34}--\eqref{38} for each natural number $M\in [2^{n_0},2^{n_{\nu_0}})$, measurable set $e\subseteq E_\varepsilon$ and $p\in[1,p_0]$ we get

\begin{equation*}
\left\|f-H\right\|_{L^{p_0}(E_\varepsilon)}
\leq\left\|f-\bar H\right\|_{L^{p_0}(E_\varepsilon)}+\left\|\sum_{k=2^{n_0}}^{2^{n_{\nu_0}}-1}\delta_k2^{-(N_0+k)}W_k\right\|_{L^{p_0}[0,1]}\leq
\end{equation*}
\begin{equation*}
\leq\sum_{k=2^{n_0}}^{2^{n_{\nu_0}}-1}\left\|\delta_k2^{-(N_0+k)}W_k\right\|_{L^{p_0}[0,1]}<2^{-N_0}<\varepsilon,
\end{equation*}

\begin{equation*}
\left\| \sum_{k=2^{n_0}}^M \delta_ka_k W_k\right\|_{L^p(e)}\leq\left\| \sum_{k=2^{n_0}}^M \delta_k\bar a_k W_k\right\|_{L^p(e)}+\left\|\sum_{k=2^{n_0}}^{M}\delta_k2^{-(N_0+k)}W_k\right\|_{L^p[0,1]}\leq
\end{equation*}
\begin{equation*}
\|f\|_{L^p(e)}+\frac{\varepsilon}{2}+\sum_{k=2^{n_0}}^{M}\left\|\delta_k2^{-(N_0+k)}W_k\right\|_{L^p[0,1]}
<\|f\|_{L^p(e)}+\frac{\varepsilon}{2}+2^{-N_0}<\|f\|_{L^p(e)}+\varepsilon
\end{equation*}
and 
\begin{equation*}
\left\| \sum_{k=2^{n_0}}^M a_k W_k\right\|_{L^1[0,1]}\leq\left\| \sum_{k=2^{n_0}}^M \bar a_k W_k\right\|_{L^1[0,1]}+\sum_{k=2^{n_0}}^{M}\left\|2^{-(N_0+k)}\right\|_{L^1[0,1]}<
\end{equation*}
\begin{equation*}
<\frac{\varepsilon}{2}+2^{-N_0}<\varepsilon.
\end{equation*}

Lemma \ref{lem3} is proved.

\end{proof}

Now with help of Lemma \ref{lem3} we will prove the main lemma of the paper, which will be used in the proof of the main theorem.

\begin{lemma}\label{lem4}
For any $\delta\in(0,1)$ there exist a weight function $0<\mu(x)\leq1,$ with $|\{x\in[0,1];\ \mu(x)=1\}|>1-\delta$, so that for any numbers $p_0>1$, $n_0\in\mathbb{N}$, $\varepsilon\in(0,1)$ and polynomial $f(x)\not \equiv 0$ in the Walsh system, one can find polynomials in the Walsh system
\begin{equation*}
P(x)=\sum_{k=2^{n_{0}}}^{2^{n}-1}a_{k}W_{k}(x)\quad\hbox{and}\quad
H(x)=\sum_{k=2^{n_{0}}}^{2^{n}-1}\delta_{k}a_{k}W_{k}(x),\quad\delta_{k}=\pm1, 0
\end{equation*}
satisfying the following conditions:
\begin{equation*}
0<a_{k+1}<a_{k}<\varepsilon ,\quad k\in\mathit{[\mathit{2}^{n_{0}},2^{n}-1)},\leqno1)
\end{equation*}
\begin{equation*}
\|f-H\|_{L^{p_0}_\mu [0,1]}<\varepsilon, \leqno2)
\end{equation*}
\begin{equation*}
\max_{2^{n_{0}}\leq M<2^{n}}\left\|\sum_{k=2^{n_{0}}}^{M}\delta_{k}
a_{k}W_{k}\right\|_{L^p_\mu[0,1]}<2||f||_{L^p_\mu[0,1]}+\varepsilon,\quad \forall p\in[1,p_0],\leqno3)
\end{equation*}
\begin{equation*}
\max_{2^{n_{0}}\leq M<2^{n}}\left\|\sum_{k=2^{n_{0}}}^{M}a_{k}
W_{k}\right\|_{L^1[0,1]}<\varepsilon.\leqno4)
\end{equation*}
\end{lemma}

\begin{proof}

Let $p_{m}\nearrow+\infty$, $\delta\in(0,1)$, $N_0 = 1$ and 
$
\{f_m(x)\}_{m=1}^{\infty},\ x\in[0,1]\ ,
$
be a sequence of all polynomials in the Walsh system with rational coefficients.
By successively applying lemma \ref{lem3}, one can find sets $E_m\subset [0,1]$ and polynomials in the Walsh system of the form
\begin{equation}
P_{m}(x)=\sum_{k={2^{N_{m-1}}}}^{2^{N_m}-1}a_{k}^{(m)}W_{k}
(x),\label{40}
\end{equation}
\begin{equation}
H_{m}(x)=\sum_{k={2^{N_{m-1}}}}^{2^{N_m}-1}\delta_{k}^{(m)}a_{k}
^{(m)}W_{k}(x),\quad \delta_{k}^{(m)}=\pm 1, 0,\label{41}
\end{equation}
which satisfy the following conditions for any natural number $m$:
\begin{equation}
|E_m|>1-\frac{1}{2^{m+1}},\label{42}
\end{equation}
\begin{equation}
0<a_{k+1}^{(m)}< a_k^{(m)}<\frac{1}{4^{N_{m-1}}}, \quad k\in[2^{N_{m-1}}, 2^{N_m}-1),\label{43}
\end{equation}
\begin{equation}
\left\|f_{m}-H_{m}\right\|_{L^{p_m}(E_m)}<\frac{1}{2^{m+2}},\label{44}
\end{equation}
\begin{equation}
\max_{2^{N_{m-1}}\leq M<{2^{N_m}}}\left\|\sum_{k=2^{N_{m-1}}}^{M}
\delta_{k}^{(m)}a_{k}^{(m)}W_{k}\right\|_{L^p(e)}<\|f_m\|_{L^p(e)}
+\frac{1}{2^{m+2}},\label{45}
\end{equation}
for any measurable set $e\subseteq E_m$ and $p\in[1,p_m]$, and
\begin{equation}
\max_{2^{N_{m-1}}\leq M<2^{N_m}}\left\|\sum_{k=2^{N_{m-1}}}^{M}a_{k}^{(m)}W_{k}\right\|_{L^1[0,1]}<\frac{1}{2^{m+2}}.\label{46}
\end{equation}

We set
\begin{equation}
\begin{cases}
\Omega_{n}={\displaystyle\bigcap_{m=n}^{+\infty}}E_{m}
,\quad n\in\mathbb{N},
\\
E=\Omega_{\widetilde n}={\displaystyle\bigcap_{m=\widetilde n
}^{+\infty}}E_m,\quad \widetilde n=[\log_{1/2}\delta]+1,
\\
B=\Omega_{\widetilde n
}\bigcup\left(\displaystyle{\bigcup_{n=\widetilde n+1}^{+\infty}}\Omega_{n}
\setminus\Omega_{n-1}\right).
\end{cases}
\label{47}
\end{equation}

It is clear (see \eqref{42} and \eqref{47}) that
\begin{equation*}
|B|=1,\quad |E|>1-\delta.
\end{equation*}

We define a function $\mu(x)$ in the following way:
\begin{equation}
\mu(x)=\begin{cases}
1,\quad x\in E\cup([0,1]\setminus B),
\\
\mu_{n},\quad \ x\in\Omega_{n}
\setminus\Omega_{n-1},\ n\geq \widetilde n + 1,
\end{cases}
\label{48}
\end{equation}
where
\begin{equation}
\mu_{n}=\frac{1}{2^{p_{n}(n+2)}}\cdot\left[\prod_{m=1}^{n}
h_{m}\right]  ^{-1},\label{49}
\end{equation}

\begin{equation*}
h_m=\max_{1\leq p\leq p_m}\left(1+\int_{0}^{1}
|f_m(x)|^{p}dx+\max_{2^{N_{m-1}}\leq M<2^{N_m}}\int_{0}^{1}\left|
\sum_{k=2^{N_{m-1}}}^{M}\delta_k^{(m)}a_{k}^{(m)}W_k(x)\right|^{p}dx\right).
\end{equation*}

It follows  from \eqref{47}--\eqref{49} that for all $m\geq \widetilde n$
\begin{equation}
\int_{[0,1]\setminus\Omega_{m}}\left|H_m(x)\right|^{p_m}
\mu(x)dx=\sum_{n=m+1}^{+\infty}\left(  \int_{\Omega_{n}\setminus\Omega_{n-1}
}\left|H_m(x)\right|^{p_m}\mu_{n}dx\right) <\label{50}
\end{equation}
\begin{equation*}
<\sum_{n=m+1}^{\infty}\frac{1}{2^{p_{n}(n+2)}h_m}\left(  \int_{0}^{1}\left|H_{m}(x)\right|^{p_m}dx\right)<\frac{1}{2^{p_m(m+2)}}.
\end{equation*}

In a similar way for all $m\geq \widetilde n$, $M\in[2^{N_{m-1}},2^{N_m})$ and $p\in[1,p_{m}]$ we have
\begin{equation}
\int_{[0,1]\setminus\Omega_{m}}\left| f_{m}(x)\right|^{p_m}
\mu(x)dx<\frac{1}{2^{p_m(m+2)}}\label{51}
\end{equation}
and 
\begin{equation}
\int_{[0,1]\setminus\Omega_{m}}\left|\sum_{k=2^{N_{m-1}}}^{M}\delta_k^{(m)}a_k^{(m)}W_k(x)\right| ^{p}\mu(x)dx< \frac{1}{2^{p(m+2)}}.\label{52}
\end{equation}

Since $\Omega_m\subset E_m$, by using conditions \eqref{44}, \eqref{47}--\eqref{51} and Jensen's inequality, for all $m\geq \widetilde n$ we obtain
\begin{equation*}
\int_{0}^{1}|f_{m}(x)-{H}_{m}(x)|^{p_{m}}\mu(x)dx=\int_{\Omega_{m}}|f_{m}(x)-{H}_{m}(x)|^{p_{m}}\mu(x)dx+
\end{equation*}
\begin{equation*}
+\int_{[0,1]\setminus\Omega_{m}}|f_{m}(x)-{H}_{m}(x)|^{p_{m}}\mu
(x)dx<\frac{1}{2^{p_{m}(m+2)}}+2\cdot2^{p_m}\frac{1}{2^{p_m(m+2)}}<
\frac{1}{2^{p_m(m-1)}},
\end{equation*}
\smallskip
or
\begin{equation}
\|f_{m}-{H}_{m}\|_{L^{p_{m}}_\mu[0,1]}<\frac{1}{2^{m-1}}.\label{53}
\end{equation}

Further, taking relations \eqref{45}, \eqref{47}--\eqref{49}, \eqref{52} and Jensen's inequality into account for all  $M\in[2^{N_{m-1}},2^{N_m})$, $p\in[1,p_{m}]$ and $m\geq \widetilde n+1$ we get
\begin{equation*}
\int_{0}^{1}\left|\sum_{k=2^{N_{m-1}}}^{M}\delta_k^{(m)}a_k^{(m)}W_k(x)\right| ^{p}\mu(x)dx=\int_{\Omega_{m}}\left|\sum_{k=2^{N_{m-1}}}^{M}\delta_k^{(m)}a_k^{(m)}W_k(x)\right|^{p}\mu(x)dx+
\end{equation*}
\begin{equation*}
+\int_{[0,1]\setminus\Omega_{m}}\left|\sum_{k=2^{N_{m-1}}}^{M}\delta_k^{(m)}a_k^{(m)}W_k(x)\right|^{p}\mu(x)dx<
\end{equation*}
\begin{equation*}
<\int_{\Omega_{\widetilde n}}\left|\sum_{k=2^{N_{m-1}}}^{M}\delta_k^{(m)}a_k^{(m)}W_k(x)\right|^{p}\mu(x)dx+
\end{equation*}
\begin{equation*}
+\sum_{n=\widetilde n+1}^{m}\mu_{n}\cdot\int_{\Omega_{n}\setminus\Omega_{n-1}}\left|
\sum_{k=2^{N_{m-1}}}^{M}\delta_k^{(m)}a_k^{(m)}W_k(x)\right|^{p}dx+\frac{1}{2^{p(m+2)}}<
\end{equation*}
\begin{equation*}
<\left(\|f_m\|_{L^p(\Omega_{\widetilde n})}+\frac{1}{2^{m+2}} \right)^p+\sum_{n=\widetilde n+1}^{m}\mu_{n}\left(\|f_m\|_{L^p(\Omega_{n}\setminus\Omega_{n-1})}+\frac{1}{2^{m+2}}\right)^p+\frac{1}{2^{p(m+2)}}\leq
\end{equation*}
\begin{equation*}
\leq2^p\left(\int_{\Omega_{\widetilde n}}|f_{m}(x)|^{p}dx+\sum_{n=\widetilde n+1}^{m}\int_{\Omega_{n} \setminus\Omega_{n-1}}|f_{m}(x)|^{p}\cdot\mu_{n}dx\right)+
\end{equation*}
\begin{equation*}
+\frac{1}{2^{p(m+2)}}\left(2^p+2^p\cdot\sum_{n=\widetilde n+1}^{m}\mu_n+1\right)<2^p\|f_m\|^p_{L^p_\mu[0,1]}+\frac{1}{2^{p(m-1)}}
\end{equation*}
or
\begin{equation}
\left\|\sum_{k=2^{N_{m-1}}}^{M}\delta_k^{(m)}a_k^{(m)}W_k\right\|_{L^p_\mu[0,1]}<2\|f_m\|_{L^p_\mu[0,1]}+\frac{1}{2^{m-1}}.\label{54}
\end{equation}

Let $n_0\in\mathbb{N}$ and $\varepsilon\in(0,1)$ be arbitrarily given. From the sequence $\{f_m(x)\}_{m=1}^{\infty}$ we choose such a function $f_{m_0}(x)$ that 
\begin{equation}
m_0>\max\left\{\tilde n,\ \log_2\frac{8}{\varepsilon}\right\}, \quad p_{m_0}>p_0,\quad 2^{N_{m_0-1}}>2^{n_0},\label{55}
\end{equation}
\begin{equation}
\|f-f_{m_{0}}\|_{L^{p_0}[0,1]}<\frac{\epsilon}{4},\label{56}
\end{equation}
and for $k\in\left[2^{n_0},2^{N_{m_0}}\right)$ set
\begin{equation}
a_k=
\begin{cases}
a_{2^{N_{m_0-1}}}^{(m_0)}+\frac{1}{2^{k+m_0}},& \hbox{when} \quad k\in\left[2^{n_0},2^{N_{m_0-1}}\right)
\\
\\
a_k^{(m_0)},& \hbox{when} \quad k\in\left[2^{N_{m_0-1}},2^{N_{m_0}}\right),
\end{cases}\label{57}
\end{equation}
\begin{equation}
\delta_k=\begin{cases}
0,& \hbox{when} \quad k\in\left[2^{n_0},2^{N_{m_0-1}}\right)
\\
\\
\delta_k^{(m_0)}=\pm1,0,& \hbox{when} \quad k\in\left[2^{N_{m_0-1}},2^{N_{m_0}}\right)
\end{cases}\label{58}
\end{equation}
and
\begin{equation*}
P(x)=\sum_{k=2^{n_0}}^{2^{N_{m_0}}-1}a_kW_k(x)=\sum_{k=2^{n_0}}^{2^{N_{m_0-1}}-1}a_kW_k(x)+P_{m_0}(x),
\end{equation*}
\begin{equation*}
H(x)=\sum_{k=2^{n_0}}^{2^{N_{m_0}}-1}\delta_ka_kW_k(x)=H_{m_0}(x).
\end{equation*}

Now it is not hard to verify that the function $\mu(x)$ and polynomials $P(x)$ and $H(x)$ satisfy all requirements of lemma \ref{lem4}. Indeed, statements 1)--3) immediately follow from \eqref{43}, \eqref{53}--\eqref{58}. Further, by using  \eqref{46}, \eqref{55}--\eqref{57} we obtain
\begin{equation*}
\max_{2^{n_0}\leq M<{2^{N_{m_0}}}}\left\|\sum_{k=2^{n_0}}^Ma_kW_k\right\|_{L^1[0,1]}\leq\max_{2^{n_0}\leq M_1<{2^{N_{m_0-1}}}}\left\|\sum_{k=2^{n_0}}^{M_1}a_kW_k\right\|_{L^1[0,1]}+
\end{equation*} 
\begin{equation*}
+\max_{{2^{N_{m_0-1}}}\leq M_2<{2^{N_{m_0}}}}\left\|\sum_{k={2^{N_{m_0-1}}}}^{M_2}a_k^{(m_0)}W_k\right\|_{L^1[0,1]}<\max_{2^{n_0}\leq M_1<{2^{N_{m_0-1}}}}\left\|\sum_{k=2^{n_0}}^{M_1}a_kW_k\right\|_{L^1[0,1]}+\frac{\varepsilon}{2}.
\end{equation*}

Let $M_1$ be an arbitrary natural number from $\left[2^{n_0},2^{N_{m_0-1}}\right)$. Then $M_1\in [2^{n_1},2^{n_1+1})$ for some $n_1\in[n_0,N_{m_0-1})$ and, considering \eqref{1.a}, we have 
\begin{equation*}
\left\|\sum_{k=2^{n_0}}^{M_1}a_kW_k\right\|_{L^1[0,1]}< a_{2^{N_{m_0-1}}}^{(m_0)}\cdot\left\|\sum_{k=2^{n_0}}^{2^{n_1}-1}W_k\right\|_{L^1[0,1]}+a_{2^{N_{m_0-1}}}^{(m_0)}\cdot2^{n_1}+
\end{equation*}
\begin{equation*}
+\sum_{k=2^{n_0}}^{M_1}\frac{1}{2^{k+m_0}}<\frac{\varepsilon}{2},
\end{equation*}
which proves the statement 4). 

Lemma \ref{lem4} is proved.

\end{proof}

\section{Proof of  theorem \ref{main}}

Let $\delta\in(0,1)$, $p_{m}\nearrow+\infty$ and $\{f_{m}(x)\}_{m=1}^{\infty},\ x\in[0,1],
$ be a sequence of all polynomials in the Walsh system with rational coefficients. By applying Lemma \ref{lem4}, we obtain a weight function  $0<\mu(x)\leq1$ with $|\{x\in[0,1],\ \mu(x)=1\}|>1-\delta$ and polynomials in the Walsh systems 
\begin{equation}
P_{m}(x)=\sum_{k={N_{m-1}}}^{{N_{m}^{{}}}-1}a_{k}^{(m)}W_{k}
(x),\label{60}
\end{equation}
\begin{equation}
H_{m}(x)=\sum_{k={N_{m-1}}}^{{N_{m}-1}}\delta_{k}^{(m)}a_{k}^{(m)}W_{k}(x),\quad \delta_{k}^{(m)}=\pm 1, 0, \label{61}
\end{equation}
which satisfy the following conditions for any natural number $m$:
\begin{equation}
\begin{cases}
0<a_{k+1}^{(1)}<a_{k}^{(1)},
\\
\\
0<a_{k+1}^{(m)}<a_{k}^{(m)}<\min\bigl\{2^{-m}, a_{{N_{m - 1}} - 1}^{(m-1)}\bigr\} & \hbox{for}\quad m>1,
\end{cases}\label{62}
\end{equation}
when $ k\in[N_{m-1},N_m-1)$,
\begin{equation}
\left\|f_{m}-H_{m}\right\|_{L^{p_m}_\mu[0,1]}<2^{-m-1},\label{63}
\end{equation}
\begin{equation}
\max_{N_{m-1}\leq M<{N_m}}\left\|\sum_{k={N_{m-1}}}^{M}
\delta_{k}^{(m)}a_{k}^{(m)}W_{k}\right\|_{L^p_\mu[0,1]}<2\|f_m\|_{L^p_\mu[0,1]}
+2^{-m},\label{64}
\end{equation}
for any $p\in [1,p_m]$, and
\begin{equation}
\max_{N_{m-1}\leq M<N_m}\left\|\sum_{k=N_{m-1}}^{M}a_{k}^{(m)}W_{k}\right\|_{L^1[0,1]}<2^{-m-1}.\label{65}
\end{equation}
From \eqref{60} and \eqref{65} it immediately follows that
\begin{equation}
\left\|\sum_{m=1}^{\infty}P_{m}\right\|_{L^1[0,1]}\leq\sum_{m=1}^{\infty}\left\|P_m\right\|_{L^1[0,1]}<+\infty.\label{66}
\end{equation}

By denoting
\begin{equation}
P_0(x)=\sum_{k=0}^{N_0-1}a_{k}W_{k}(x),\label{67}
\end{equation}
where coefficients $a_k,\ k\in[0,N_0)$, are arbitrary monotonically decreasing positive numbers with $a_{N_0-1}>a_{N_0}^{(1)}$, we define a function $g(x)$ and a series $\sum_{k=0}^{\infty}a_kW_k(x)$ as follows:
\begin{equation}
g(x)=\sum_{m=0}^{\infty}P_{m}(x),\label{68}
\end{equation}
\begin{equation}
a_{k}=a_{k}^{(m)},\ \mbox{when} \ k\in[N_{m-1}, N_m),\ m\in\mathbb{N},\label{69}
\end{equation}
and $a_k$ are coefficients in $P_0(x)$ (see \eqref{67}), when $k\in[0,N_0)$. 
By using \eqref{62}, \eqref{65}--\eqref{69} we conclude that the series $\sum_{k=0}^{\infty}a_kW_k(x)$ converges to $g\in L^1[0,1]$ in $L^1[0,1]$ metric, and $a_{k}=\int_{0}^{1}g(t)W_{k}(t)dt \searrow0$.

Let $p\geq1$ and let $f\in L_{\mu}^{p}(0,1)$. We choose such a polynomial $f_{\nu_{1}}(x)$ from the sequence $\{f_{m}(x)\}_{m=1}^{\infty}$ that
\begin{equation}
\left\|f-f_{\nu_{1}}\right\|_{L^p_\mu[0,1]}<2^{-2}\quad\hbox{and} \quad p_{\nu_{1}}>p.\label{70}
\end{equation}

By denoting  
\begin{equation*}
\delta_{k}= 
\begin{cases}
\delta_{k}^{({\nu_{1}})}=\pm 1,0,&  \hbox{when}\quad k\in[N_{{\nu}_1-1},N_{\nu_1}),
\\
0, &  \hbox{when}\quad k\in[0,N_{{\nu}_1-1}),
\end{cases}
\end{equation*}
and taking into account \eqref{61}, \eqref{63}, \eqref{64} and \eqref{70}, we have
\begin{equation*}
\left\|f-\sum_{k=0}^{N_{\nu_1}-1}\delta_{k}a_{k}
W_{k}\right\|_{L^p_\mu[0,1]}\leq\left\|f-f_{\nu_{1}}\right\|_{L^p_\mu[0,1]}+\left\|f_{\nu_1}-H_{\nu_1}\right\|_{L^{p_{\nu_1}}_\mu[0,1]}<
\end{equation*}
\begin{equation*}
<2^{-2}+2^{-\nu_1-1}<2^{-1},
\end{equation*}
and
\begin{equation*}
\max_{N_{{\nu}_1-1}\leq M<N_{{\nu}_1}}\left\|\sum_{k=N_{\nu_1-1}}^{M}\delta_{k}a_{k}
W_{k}\right\|_{L^p_\mu[0,1]}<2\|f_{\nu_{1}}\|_{L^p_\mu[0,1]}+2^{-\nu_1}.
\end{equation*}

Assume that for $q>1$ numbers $\nu_{1}<\nu_{2}<\dots<\nu_{q-1}$ and $\{\delta_k=\pm 1, 0\}_{k=0}^{N_{\nu_{q-1}}-1}$ are already chosen, so that for each natural number $j\in[1,q-1]$ the following conditions hold:
\begin{equation*}
\delta_{k}= 
\begin{cases}
\delta_{k}^{({\nu_{j}})}=\pm 1, 0,&  \hbox{when}\quad k\in[N_{{\nu}_j-1},N_{\nu_j}),
\\
0, &  \hbox{when}\quad k\notin\bigcup_{j=1}^{q-1}[N_{{\nu}_j-1},N_{\nu_j}),
\end{cases}
\end{equation*}
\begin{equation}
\left\|f-\sum_{k=0}^{N_{\nu_j}-1}\delta_{k}a_{k}
W_{k}\right\|_{L^p_\mu[0,1]}<2^{-j},\label{71}
\end{equation}
\begin{equation*}
\max_{N_{{\nu}_j-1}\leq M<N_{{\nu}_j}}\left\|\sum_{k=N_{\nu_j-1}}^{M}\delta_ka_k
W_{k}\right\|_{L^p_\mu[0,1]}<2\|f_{\nu_{j}}\|_{L^p_\mu[0,1]}+2^{-\nu_j}.
\end{equation*}

We choose a function $f_{\nu_{q}}(x)$ from the sequence $\{f_{m}(x)\}_{m=1}^{\infty}$ with $\nu_{q}>\nu_{q-1}$ so that
\begin{equation}
\left\|f-\sum_{k=0}^{N_{\nu_{q-1}}-1}\delta_{k}a_{k}
W_{k}(x)-f_{\nu_q}\right\|_{L^p_\mu[0,1]}<2^{-q-1},\label{72}
\end{equation}
and define
\begin{equation}
\delta_{k}= 
\begin{cases}
\delta_{k}^{({\nu_{q}})}=\pm 1, 0,&  \hbox{when}\quad k\in\bigr[N_{{\nu}_q-1},N_{\nu_q}\bigr),
\\
0, &  \hbox{when}\quad  k\notin\bigcup_{j=1}^{q}[N_{{\nu}_j-1},N_{\nu_j}).
\end{cases}\label{73}
\end{equation}

Taking into account \eqref{61}, \eqref{63}, \eqref{72} and \eqref{73}, we get
\begin{equation}
\left\|f-\sum_{k=0}^{N_{\nu_q}-1}\delta_{k}a_{k}
W_{k}\right\|_{L^p_\mu[0,1]}\leq\label{74}
\end{equation}
\begin{equation*}
\leq\left\|f-\sum_{k=0}^{N_{\nu_{q-1}}-1}\delta_{k}a_{k}
W_{k}-f_{\nu_q}\right\|_{L^p_\mu[0,1]}+\left\|f_{\nu_q}-H_{\nu_q}\right\|_{L^{p_{\nu_q}}_\mu[0,1]}<
\end{equation*}
\begin{equation*}
<2^{-q-1}+2^{-\nu_q-1}< 2^{-q}.
\end{equation*}

Further, from \eqref{71} and \eqref{72} we have
\begin{equation*}
\|f_{\nu_{q}}\|_{L^p_\mu[0,1]}<\left\|f-\sum_{k=0}^{N_{\nu_{q-1}}-1}\delta_{k}a_{k}
W_{k}-f_{\nu_q}\right\|_{L^p_\mu[0,1]}+
\end{equation*}
\begin{equation*}
+\left\|f-\sum_{k=0}^{N_{\nu_{q-1}}-1}\delta_{k}a_{k}
W_{k}\right\|_{L^p_\mu[0,1]}<2^{-q-1} + 2^{-q+1}<2^{-q+2}.
\end{equation*}
Thus, from \eqref{64} and \eqref{73} it follows that for each natural number $M\in[N_{{\nu}_q-1},N_{{\nu}_q})$
\begin{equation}
\left\|\sum_{k=N_{{\nu}_q-1}}^{M}\delta_ka_k
W_{k}\right\|_{L^p_\mu[0,1]}<2\|f_{\nu_{q}}\|_{L^p_\mu[0,1]}+2^{-\nu_q}<2^{-q+4}.\label{75}
\end{equation}

Clearly, by using induction one can determine growing sequence of indexes $\{\nu_q\}_{q=1}^{+\infty}$ and numbers $\{\delta_k=\pm 1,0\}_{k=0}^{+\infty}$ so that conditions \eqref{73}--\eqref{75} hold for any $q\in\mathbb{N}$. Hence,  we obtain a series
\begin{equation}
\sum_{k=0}^{+\infty}\delta_{k}a_{k}W_{k}(x),\quad \delta_k=\pm 1, 0, \label{76}
\end{equation}
which converges to $f$ in $L^p_\mu[0,1]$ metric. Indeed, from \eqref{74} it follows that the subsequence $\{S_{N_{\nu_q}}(x)\}_{q=1}^{+\infty}$ of its partial sums
\begin{equation*}
S_N(x)\equiv\sum_{k=0}^{N-1}\delta_{k}a_{k}W_{k}(x),\quad N=1,2,\dots ,
\end{equation*}
converges to $f$ in $L^p_\mu[0,1]$ metric, and \eqref{75} provides the convergence of the whole sequence $S_N(x)$.

The theorem \ref{main} is proved.

\bibliographystyle{amsplain}

\end{document}